\documentclass[11pt]{article} 
\usepackage{graphicx}
\usepackage{amsmath, amssymb, amsthm} 
\usepackage{verbatim}                  
\usepackage[alphabetic]{amsrefs}       

\theoremstyle{plain} 
\newtheorem{theorem}{Theorem}
\newtheorem{lemma}[theorem]{Lemma}
\newtheorem{corollary}[theorem]{Corollary}
\newtheorem{proposition}[theorem]{Proposition}

\newtheorem*{claim}{Claim}
\theoremstyle{definition} 
\newtheorem{definition}[theorem]{Definition}
\newtheorem{example}[theorem]{Example}

\newtheorem{remark}[theorem]{Remark}
\newtheorem*{acknowledgements}{Acknowledgements}
\title{Multi-way expansion constants and partitions of a graph}
\author{Mamoru Tanaka}
\date{} 
\begin{document}

\maketitle

\footnote{Key words and key phrases. \textit{Multi-way expansion constant, Eigenvalues of Laplacian, Partition, sequence of expanders, Coarse embeddability}.}

\begin{abstract}
In this paper, we consider a relation between $k$-way expansion constant of a finite graph  
and the expansion constants of subgraphs in a $k$-partition of the graph. 
Using this relation, 
we show that a sequence of finite graphs which have uniformly bounded $k+1$-way expansion constants and uniformly bounded degrees can be divided into $k$ or fewer sequences of expanders. 
Furthermore, we prove that any such sequence of finite graphs is not coarsely embeddable into any Hilbert space. 
\end{abstract}


\section{Introduction} 
We assume graphs are non-oriented, don't have loops and multiple edges.  

Let $G=(V,E)$ be a finite graph in what follows.  
For a graph $H$, $V_H$ denotes the vertex set of $H$ and $E_H$ denotes the edge set of $H$ if they are unspecified. 
The {\it expansion constant} of $G$ is  
$$h(G):= \min_{F\subset V} \left\{ \frac{|\partial F|}{|F|}: 1\le|F|\le \frac{|V|}{2} \right\} $$
where $\partial F$ is the set of the edges connecting $F$ and $V-F$. 
The expansion constant of a graph represents a strength of connection of the graph.
Lee-Gharan-Trevisan \cite{MR2961569} generalized expansion constant as 
\begin{eqnarray*}
h_k(G):= \min\left\{ \max_{i=1,2,\dots, k} \frac{|\partial V^i|}{|V^i|} : V=\bigsqcup_{i=1}^kV^i\text{(disjoint union)}, V^i\not= \emptyset \right\} 
\end{eqnarray*}
for $1\le k \le |V|$.
We call it the {\it $k$-way expansion constant} of $G$.   
Note that $h_1(G)=0$ and $h_2(G)=h(G)$.

The $k$-way expansion constant has the property that for each $k\ge 2$, $h_k(G)=0$ and $h_{k+1}(G)>0$ if and only if the number of connected components of $G$ is $k$. 
Moreover, if the number of connected components of $G$ is $k$, then $h_{k+1}(G) = \min_{i=1,2,\dots ,k} h(G^i)$ where $G^i$'s are connected components of $G$. 
We consider such a property for $k$-partitions of $G$. 
An {\it induced subgraph} of $G$ is a graph $H$ satisfying $V_H \subset V$ and 
$E_H=\{xy\in E: x,y\in V_H\}$. 
Hence an induced subgraph is determined by its vertex set. 
A {\it $k$-partition} of $G$ is a family of induced subgraphs $\{ G^i=(V^i,E^i)\}_{i=1}^k$ such that $V$ is the disjoint union of $V^i$'s. 
We can easily show that $h_{k+1}(G) \ge \min_{i=1,2,\dots ,k}h(G^i)$ for any $k$-partition $\{G^i\}_{i=1}^k$ of $G$ (Lemma \ref{lem:1}). 
In section \ref{sec:EP}, we prove that if $h_{k+1}(G)/3^{k+1}>h_k(G)$ for some $k$, then there is a $k$-partition $\{G^i=(V^i,E^i)\}_{i=1}^k$ of $G$ such that 
\begin{eqnarray*}
\frac{h_{k+1}(G)}{3^{k+1}}\le \min_{i=1,2,\dots ,k}h(G^i), \ \ \ 
\max_{i=1,2,\dots, k} \frac{|\partial V^i|}{|V^i|} \le 3^kh_k(G)
\end{eqnarray*}
(Theorem \ref{thm:l-sub}). 

A sequence of expanders has an important role in computer science, group theory, geometry and topology (cf.\ \cite{MR2247919}). 
A {\it sequence of expanders} is a sequence of finite graphs $\{G_n=(V_n,E_n)\}_{n=1}^\infty$ such that 
(i) $\lim_{n \to \infty}|V_n|=\infty;$ 
(ii) $\sup_{n\in \mathbb{N}}\deg(G_n) < \infty ;$ 
(iii) $\inf_{n\in \mathbb{N}}h(G_n)>0$. 
Here the degree $\deg(x)$ of $x\in V$ is the number of edges such that $x$ is an end point of it, and $\deg(G)$ is the maximum number of the degrees.  
It is natural to consider more general sequence of finite graphs ({\it a sequence of multi-way expanders}) $\{G_n=(V_n,E_n)\}_{n=1}^\infty$ with the following conditions: 
(i) $\lim_{n \to \infty}|V_n|=\infty ;$ 
(ii) $\sup_{n\in \mathbb{N}}\deg(G_n) < \infty ;$ 
(iii)' $\inf_{n\in \mathbb{N}}h_{k+1}(G_n)>0$ for some $k\ge 1$. 
In section \ref{sec:Exp}, we prove that such a sequence of finite graphs can be divided into $k$ or fewer sequences of expanders (Corollary \ref{cor:1}).

We can endow a graph $G$ with the {\it path metric} $d_G(x,y)$ between vertices $x$ and $y$, which is the minimum number of edges in any path connecting $x$ and $y$. 
It is well-known that a sequence of expanders is not coarsely embeddable into any Hilbert space (\cite{MR1978492}). 
Here the coarse embeddability (also referred to as uniform embeddability) is an embeddability between metric spaces, 
which was introduced by Gromov in \cite{MR1253544}, and will be defined in section  \ref{sec:CE}. 
It is shown by Yu in \cite{MR1728880} that if an infinite graph with bounded degree is coarsely embeddable into a Hilbert space, then the graph satisfies the coarse Baum-Connes conjecture. 
An infinite graph which is not coarsely embeddable into any Hilbert space is an obstruction to verify the coarse Baum-Connes conjecture using Yu's result. 
It is easy to see that if a sequence of expanders is coarsely embeddable into an infinite graph, then the infinite graph is not coarsely embeddable into any Hilbert space. 
On the other hand, in section \ref{sec:CE}, we prove that a sequence of graphs with uniformly bounded degrees which has a sequence of expanders as induced subgraphs is not coarsely embeddable into any Hilbert space. 
Using this, we show that a sequence of finite graphs satisfying (i), (ii) and (iii)' is also not coarsely embeddable into any Hilbert space. 

The eigenvalues of the Laplacian on a graph is also an important constant associated to the graph. 
The Laplacian $\Delta_G$ is an operator on $\mathbb{R}^V$ defined by 
$$\Delta_G f(x) := f(x)\deg(x)-\sum_{xy\in E}f(y)$$ 
for $f\in \mathbb{R}^V$ and $x\in V$, where $\mathbb{R}^V$ is the set of the real-valued functions on $V$.
Let $\lambda _1(G)\le \lambda _2(G)\le \dots \le \lambda_{|V|}(G)$ denote the eigenvalues of $\Delta_G$ according to largeness. 
It is known that $\lambda _k(G)=0$ and $\lambda _{k+1}(G)>0$ if and only if the number of connected components of $G$ is $k$. 
This suggests that eigenvalues of the Laplacian also represents some strength of connectivity of the graph, 
so $\lambda _2(G)$ is called the {\it algebraic connectivity} of $G$. 
We can also easily show that $\min_{i=1,2,\dots ,k}\lambda_2(G^i)\le \lambda _{k+1}(G)$ for any $k$-partition $\{G^i\}_{i=1}^k$ of $G$ (Lemma \ref{lem:2}). 
It is shown in \cite{MR782626}, \cite{MR743744}, (cf. \cite{MR1989434})
that the expansion constant is estimated above and below using $\lambda_2(G)$:
\begin{equation}\label{eq:h=lambda}
\frac{\lambda_2(G)}{2}\le h(G)\le \sqrt{2\deg(G)\lambda_2(G)}.
\end{equation}
As the generalization of this inequality, Lee-Gharan-Trevisan \cite{MR2961569} prove that there is a constant $C>0$ such that  
\begin{equation}
\frac{\lambda_k(G)}{2\deg(G)}\le h_k(G)\le Ck^2\deg(G) \sqrt{\lambda_k(G)} \label{eq:hk=lambdak}
\end{equation}
for every connected graph $G$ and every $k=1,2,\dots, |V|$, 
where their original statement is described using the eigenvalues of the normalized Laplacian on a graph and weighted 
 $k$-way expansion constants, and it doesn't need the connectivity of the graph and we can omit the dependence on $\deg(G)$. 
Using their result, we obtain that a sequence of finite graphs $\{G_n=(V_n,E_n)\}_{n=1}^\infty$ with the following conditions can be divided into $k$ or fewer sequences of expanders: 
(i) $\lim_{n \to \infty}|V_n|=\infty ;$ 
(ii) $\sup_{n\in \mathbb{N}}\deg(G_n) < \infty ;$ 
(iii)'' $\inf_{n\in \mathbb{N}}\lambda_{k+1}(G_n)>0$ for some $k\ge 1$. 
A similar result for Riemannian manifolds was given by Funano and Shioya \cite{MR3061776}: 
A sequence of closed weighted Riemannian manifolds whose $k+1$-th eigenvalues diverges to $\infty$ is a union of $k$ L${\rm \acute{e}}$vy families.

\section{Partitions and $k$-way expansion constants}\label{sec:EP}

For an induced subgraph $H$ of $G$, $G-H$ denotes the induced subgraph of $G$ whose vertex set is $V-V_H$, and we call it the {\it complement subgraph} of $H$ in $G$. 
For an induced subgraph $H_1$ of $G$ and an induced subgraph $H_2$ of $H_1$, let $\partial_{H_1} V_{H_2} := \{ xy\in E_{H_1}| x\in V_{H_2}, y\in V_{H_1-H_2}\}$.

\begin{lemma}\label{lem:1}
For any $k$-partition $\{G^i=(V^i,E^i)\}_{i=1}^k$ of $G$
\begin{equation}
h_{k+1}(G) \ge \min_{i=1,2,\dots ,k}h(G^i) .\label{eq:hh}
\end{equation}
\end{lemma}

\begin{proof}
Take a $(k+1)$-partition $\{\tilde{G}^j=(\tilde{V}^j,\tilde{E}^j)\}_{j=1}^{k+1}$ of $G$ such that 
$$h_{k+1}(G) = \max_{j=1,2,\dots, k+1} \frac{|\partial \tilde{V}^j|}{|\tilde{V}^j|}.$$
Let $V^{i,j}:= V^i\cap \tilde{V}^j$, then $V^i=\sqcup_{j=1}^{k+1}V^{i,j}$. 
The number of the sets in $\{V^{i,j}\}_{i,j}$ with $|V^{i,j}| > |V^i|/2$ is not larger than $k$.
So there is $j_0$ such that $|V^{i,j_0}| \le |V^i|/2$ for all $i=1,2,\dots,k$.
And there is $i_0$ such that 
\begin{eqnarray*}
\frac{|\partial \tilde{V}^{j_0}|}{|\tilde{V}^{j_0}|} 
= \frac{|\sqcup_{i=1}^{k}(\partial V^{i,j_0}\cap \partial \tilde{V}^{j_0})|}{|\sqcup_{i=1}^{k}V^{i,j_0}|} 
\ge \frac{|\partial V^{i_0,j_0}\cap \partial \tilde{V}^{j_0}|}{|V^{i_0,j_0}|}.
\end{eqnarray*} 
Since  $|V^{i_0,j_0}| \le |V^{i_0}|/2$, we have 
$$ 
h_{k+1}(G) 
= \max_{j=1,\dots, k+1} \frac{|\partial \tilde{V}^j|}{|\tilde{V}^j|}
\ge \frac{|\partial \tilde{V}^{j_0}|}{|\tilde{V}^{j_0}|} 
\ge \frac{|\partial V^{i_0,j_0}\cap \partial \tilde{V}^{j_0}|}{|V^{i_0,j_0}|}
= \frac{|\partial_{G^{i_0}} V^{i_0,j_0}|}{|V^{i_0,j_0}|}
\ge h(G^{i_0}) 
\ge \min_{i=1,\dots ,k}h(G^i).
$$ 
\end{proof}

\begin{theorem}\label{thm:l-sub}
If $h_{k+1}(G)/3^{k+1}> h_k(G) $ for some $k$, 
then there exists a $k$-partition $\{G^i=(V^i,E^i)\}_{i=1}^k$ of $G$ satisfying
\begin{eqnarray}
\frac{h_{k+1}(G)}{3^{k+1} } \le  \min_{i=1,2,\dots ,k}h(G^i), \ \ \  
\max_{i=1,2,\dots, k} \frac{|\partial V^i|}{|V^i|} \le 3^kh_k(G).
\label{ineq:thm-2}
\end{eqnarray}
\end{theorem}

\begin{proof}
We may assume $k\ge 2$.
First, we divide $G$ into $k$ induced subgraphs inductively. 

\noindent 
(i) Take an induced subgraph $H^0$ of $G$ such that $$\frac{|\partial_G V_{H^0}|}{|V_{H^0}|}=h(G) \text{\ \ \  and \ \ \ } |V_{H^0}|\le \frac{|V|}{2}, $$ and set $H^1:= G-H^0$. 

\noindent 
(ii) Let $h(H^{i_1})\le h(H^{j_1})$ where $i_1,j_1\in \{0,1\}$ and $i_1\not=j_1$. 
Then take an induced subgraph $H^{i_10}$ of $H^{i_1}$ such that 
$$\frac{|\partial_{H^{i_1}} V_{H^{i_10}}|}{|V_{H^{i_10}}|}=h(H^{i_1}) \text{\ \ \  and \ \ \ } |V_{H^{i_10}}|\le \frac{|V_{H^{i_1}}|}{2} ,$$ 
and set $H^{i_11}:= H^{i_1}-H^{i_10}$. 

\noindent 
(iii) If $h(H^{j_1})\le \min \{h(H^{i_10}),h(H^{i_11})\}$, then take an induced subgraph $H^{j_10}$ of $H^{j_1}$ such that $$\frac{|\partial_{H^{j_1}} V_{H^{j_10}}|}{|V_{H^{j_10}}|}=h(H^{j_1}) \text{\ \ \  and \ \ \ } |V_{H^{j_10}}|\le \frac{|V_{H^{j_1}}|}{2}, $$ and set $H^{j_11}:= H^{j_1}-H^{j_10}$. 
Otherwise set $h(H^{i_1i_2})\le h(H^{i_1j_2})$ where $i_2,j_2\in \{0,1\}$ and $i_2\not = j_2$. 
Take an induced subgraph $H^{i_1i_20}$ of $H^{i_1i_2}$ such that 
$$\frac{|\partial_{H^{i_1i_2}} V_{H^{i_1i_20}}|}{|V_{H^{i_1i_20}}|}=h(H^{i_1i_2}) \text{\ \ \  and \ \ \ } |V_{H^{i_1i_20}}|\le \frac{|V_{H^{i_1i_2}}|}{2} ,$$ and set $H^{i_1i_21}:= H^{i_1i_2}-H^{i_1i_20}$.

\noindent 
In this manner, we inductively divide an undivided $H^{a_1a_2\dots a_m}$ which has the minimum expansion constant among the undivided induced subgraphs in $\{H^{a_1a_2\dots a_n}\}$, into an induced subgraph $H^{a_1a_2\dots a_m0}$ which attains the expansion constant $h(H^{a_1a_2\dots a_m})$ and the complement subgraph $H^{a_1a_2\dots a_m1}$ in $H^{a_1a_2\dots a_m}$, where $a_l\in \{0,1\}$ ($l=1,2,\dots ,m$).  
Repeat this procedure until the number of the undivided induced subgraphs in $\{H^{a_1a_2\dots a_n}\}$ becomes $k$. 
Consequently we divided $G$ into $k$ induced subgraphs, and we have hierarchical sequences of induced subgraphs 
$$G \supset H^{a_1} \supset H^{a_1a_2} \supset \dots \supset H^{a_1a_2\dots a_n}.$$
Let $D^1:=G$ and $D^i$ denote the induced subgraph of $G$ which was devided in the $i$-th step of the induction for $i=2, \dots ,k-1$.
Let  $\{G^i=(V_i,E_i)\}_{i=1}^k$ denotes the set of the undivided induced subgraphs in $\{H^{a_1a_2\dots a_n}\}$, that is $\{ G^i \}_{i=1}^k = \{ H^{a_1a_2\dots a_n} \} - \{D^i\}_{i=1}^{k-1}$. 
The set $\{G^i\}_{i=1}^k$ is a partition of $G$. 

Next, we will show that this partition $\{G^i\}_{i=1}^k$ satisfies the inequalities (\ref{ineq:thm-2}). 
Fix an induced subgraph $H^{a_1a_2\dots a_s}$ with $s \ge 2$. 
Let $H^{(0)}:=G$ and $H^{(r)}:=H^{a_1a_2\dots a_r}$ for $r\le s$.  
We denote $\partial_{(p,q)}:= \partial_{H^{(p)}} V_{H^{(q)}}$ for $p<q\le s$. 
%
%
\begin{claim}
For $p$ and $q:=p+1<s$,  $|\partial_{(p,s)}-\partial_{(q,s)}| \le 2|\partial_{(q,s)}| + 2h(H^{(p)})|V_{H^{(s)}}| $.
\end{claim}
Since $\partial_{(p,q)} = (\partial_{(p,q)} -\partial_{(q,s)}) \sqcup \partial_{(q,s)} $ and 
$\partial_{H^{(p)}} V_{H^{(q)}-H^{(s)}} = \partial_{(q,s)} \sqcup (\partial_{(p,q)}-\partial_{(p,s)} )$, 
we get
\begin{eqnarray*}
h(H^{(p)}) 
=\frac{|\partial_{(p,s)}-\partial_{(q,s)}| + |\partial_{(p,q)}-\partial_{(p,s)} |}{\min \left\{ |V_{H^{(q)}}|, |V_{H^{(p)}-H^{(q)}}| \right\}}
\le \frac{|\partial_{(q,s)}| + |\partial_{(p,q)}-\partial_{(p,s)}|}{\min \left\{ |V_{H^{(q)}-H^{(s)}}|,|V_{H^{(p)}-H^{(q)}}\sqcup V_{H^{(s)}}|\right\}}  . 
\end{eqnarray*}
Hence we obtain 
\begin{eqnarray}\nonumber
\hspace{-2mm}&&\hspace{-2mm}|\partial_{(p,s)}-\partial_{(q,s)}| \min \left\{ |V_{H^{(q)}-H^{(s)}}|,|V_{H^{(p)}-H^{(q)}}\sqcup V_{H^{(s)}}|\right\} \\ \nonumber
\hspace{-2mm}&\le &\hspace{-2mm} |\partial_{(q,s)}| \min \left\{ |V_{H^{(q)}}|, |V_{H^{(p)}-H^{(q)}}| \right\}
\\ \hspace{-2mm}&&\hspace{-2mm} + |\partial_{(p,q)}-\partial_{(p,s)}| \left( \min \left\{ |V_{H^{(q)}}|, |V_{H^{(p)}-H^{(q)}}| \right\}
- \min \left\{ |V_{H^{(q)}-H^{(s)}}|,|V_{H^{(p)}-H^{(q)}}\sqcup V_{H^{(s)}}|\right\}\right) \label{eq:11}
. 
\end{eqnarray}
If $|V_{H^{(q)}-H^{(s)}}|\le |V_{H^{(p)}-H^{(q)}}\sqcup V_{H^{(s)}}|$, then (\ref{eq:11}) implies 
\begin{eqnarray*}
|\partial_{(p,s)}-\partial_{(q,s)}| 
&\le & |\partial_{(q,s)}| \frac{\min \left\{ |V_{H^{(q)}}|, |V_{H^{(p)}-H^{(q)}}| \right\}}{|V_{H^{(q)}-H^{(s)}}|} 
\\ &&+ |\partial_{(p,q)}-\partial_{(p,s)}|\frac{\min \left\{ |V_{H^{(q)}}|, |V_{H^{(p)}-H^{(q)}}| \right\}- |V_{H^{(q)}-H^{(s)}}|}{|V_{H^{(q)}-H^{(s)}}| }   \\ 
&\le & |\partial_{(q,s)}|\frac{|V_{H^{(q)}}|}{|V_{H^{(q)}-H^{(s)}}| }  
+ |\partial_{(p,q)}-\partial_{(p,s)}|\frac{|V_{H^{(s)}}|}{|V_{H^{(q)}-H^{(s)}}| } . 
\end{eqnarray*}
Otherwise we have $ |V_{H^{(p)}-H^{(q)}}| \le |V_{H^{(q)}}|$, and hence (\ref{eq:11}) implies $|\partial_{(p,s)}-\partial_{(q,s)}| \le |\partial_{(q,s)}| $. 
To summarise these two cases we have
$$ |\partial_{(p,s)}-\partial_{(q,s)}| \le |\partial_{(q,s)}|\frac{|V_{H^{(q)}}|}{|V_{H^{(q)}-H^{(s)}}| }  
+ |\partial_{(p,q)}-\partial_{(p,s)}|\frac{|V_{H^{(s)}}|}{|V_{H^{(q)}-H^{(s)}}| } .$$ 
If $|V_{H^{(s)}}|\le |V_{H^{(q)}-H^{(s)}}|$, then we have $|V_{H^{(q)}}| \le 2|V_{H^{(q)}-H^{(s)}}|$, and hence 
\begin{eqnarray*}
|\partial_{(p,s)}-\partial_{(q,s)}|
&\le& 2|\partial_{(q,s)}| + 2\frac{|\partial_{(p,q)}-\partial_{(p,s)}|}{|V_{H^{(q)}}| } |V_{H^{(s)}}| 
\\ &\le& 2|\partial_{(q,s)}| + 2\frac{|\partial _{H^{(p)}}V_{H^{(q)}}|}{\min \{ |V_{H^{(p)}-H^{(q)}}|, |V_{H^{(q)}}|\} }|V_{H^{(s)}}| 
\\ &=& 2|\partial_{(q,s)}| + 2h(H^{(p)})|V_{H^{(s)}}|.
\end{eqnarray*}
Otherwise we have $|V_{H^{(q)}}| \le 2|V_{H^{(s)}}|$, and hence  
\begin{eqnarray*}
|\partial_{(p,s)}-\partial_{(q,s)}| 
\le |\partial_{(p,q)}| 
\le h(H^{(p)})|V_{H^{(q)}}|
\le 2h(H^{(p)})|V_{H^{(s)}}|.
\end{eqnarray*}
To summarise these two cases we have $|\partial_{(p,s)}-\partial_{(q,s)}| \le 2|\partial_{(q,s)}| + 2h(H^{(p)})|V_{H^{(s)}}| $. 

Let $\epsilon := \max_{i=1, 2, \dots, k-1}\{ h(D^i)\}$. 
For $p<s-1$, since $H^{(p)}\in \{D^i\}_i$, we have 
\begin{eqnarray}\label{eq:3}
|\partial_{(p,s)}-\partial_{(p+1,s)}| \nonumber
&\le & 2|\partial_{(p+1,s)}| + 2h(H^{(p)})|V_{H^{(s)}}| \nonumber
\\ &\le &  2\sum_{r=p+1}^{s-1}|\partial_{(r,s)}-\partial_{(r+1,s)}| + 2\epsilon |V_{H^{(s)}}|. 
\end{eqnarray}
Because $H^{(s-1)}\in \{D^i\}_i$,   
\begin{eqnarray}\label{eq:4}
|\partial_{(s-1,s)}| 
= h(H^{(s-1)})\min \{ |V_{H^{(s)}}|, |V_{H^{(s-1)}- H^{(s)}}| \}  
\le \epsilon |V_{H^{(s)}}| .
\end{eqnarray}
Using (\ref{eq:3}) and (\ref{eq:4}), an easy computation implies  
\begin{eqnarray*}
|\partial_{(p,s)}-\partial_{(p+1,s)}|
\le 4\cdot 3^{s-2-p}\epsilon |V_{H^{(s)}}| 
\end{eqnarray*}
for $s\ge 1$ and $0\le p\le s-1$. 
Hence 
\begin{eqnarray*}
|\partial V_{H^{(s)}}| 
= \sum_{r=0}^{s-1}|\partial_{(r,s)}-\partial_{(r+1,s)}|
\le \sum_{r=0}^{s-1} 4\cdot 3^{s-2-r}\epsilon |V_{H^{(s)}}|
\le 3^s\epsilon |V_{H^{(s)}}|
\end{eqnarray*}
for $s\ge 1$. 
On the other hand, because of the inequality (\ref{eq:hh}) and the choice of $D^i$ in the $i$-th step, we have $h(D^i)\le h_{i+1}(G)\le h_k(G)$ for $i=1,2,\dots ,k-1$. 
Therefore 
\begin{eqnarray*}
\max_{i=1,2,\dots ,k}\frac{|\partial V^i|}{|V^i|}
\le 3^k \epsilon 
= 3^k \max_{i=1, 2, \dots, k-1}\{ h(D^i)\} 
\le 3^k h_k(G)
\end{eqnarray*}

To show the second inequality in (\ref{ineq:thm-2}), we divide again an induced subgraph $D^k $ in $\{G^i\}_{i=1}^k$ which has the minimum expansion constant among $\{G^i\}_{i=1}^k$, into an induced subgraph which attains the expansion constant $h(D^k)$ and the complement subgraph. 
Consequently we divided $G$ into $k+1$ induced subgraphs. 
We denote it as $\{\tilde{G}^i=(\tilde{V}^i,\tilde{E}^i)\}_{i=1}^{k+1}$. 
Then we can prove 
\begin{eqnarray*}
\max_{i=1,2,\dots ,k,k+1}\frac{|\partial \tilde{V^i}|}{|\tilde{V^i}|}
\le 3^{k+1}\tilde{\epsilon}  
\end{eqnarray*}
for $\tilde{\epsilon} := \max_{i=1, 2, \dots, k}\{ h(D^i)\}$ as above. 
Since $\max_{i=1, 2, \dots, k-1}\{ h(D^i)\} \le h_k(G)$ and $h(D^k)=\min_{i=1,2,\dots ,k}h(G^i)$, 
using the assumption $h_{k+1}(G)> 3^{k+1}h_k(G)$, we have 
$$ 
h_{k+1}(G) 
\le \max_{i=1,\dots ,k,k+1}\frac{|\partial \tilde{V^i}|}{|\tilde{V^i}|}
\le 3^{k+1}\tilde{\epsilon}  
\le 3^{k+1}\max\left\{h_k(G), \min_{i=1,\dots ,k}h(G^i)\right\}
\le 3^{k+1} \min_{i=1,\dots ,k}h(G^i).
$$ 
\end{proof}

\section{Sequence of multi-way expanders}\label{sec:Exp}

A {\it sequence of multi-way expanders} is a sequence of finite graphs $\{G_n=(V_n,E_n)\}_{n=1}^\infty$ such that 
(i) $|V_n|\to \infty$ as $n\to \infty;$ 
(ii) $\sup_{n\in \mathbb{N}}\deg(G_n) <\infty ;$  
(iii)' $\inf_{n\in \mathbb{N}} h_{k+1}(G_n)>0$ for some $k\in \mathbb{N}$. 

\begin{example}
For sequences of expanders $\{G_n^i\}_{n=1}^\infty$ $(i=1,2,\dots ,k)$, let $\{G_n=(V_n,E_n)\}_{n=1}^\infty$ be a sequence of finite graphs such that $\sup_{n\in \mathbb{N}}\deg(G_n) <\infty $ and $\{G_n^i\}_{i=1}^k$ is a $k$-partition of $G_n$ for each $n$. 
Then using Lemma \ref{lem:1}, $\inf_{n\in \mathbb{N}} h_{k+1}(G_n)>0$, and hence $\{G_n^i\}_{i=1}^k$ is a sequence of multi-way expanders. 
\end{example}

\begin{corollary}\label{cor:1}
Let $\{G_n\}_{n=1}^\infty$ be a sequence of multi-way expanders. 
Then there are an subsequence $\{G_m\}_{m=1}^\infty$ of $\{G_n\}_{n=1}^\infty$ and induced subgraphs $H_m$ of $G_m$  for all $m$ such that $\{H_m\}_{m=1}^\infty$ is a sequence of expanders.  

Furthermore, if each graph in $\{G_n\}_{n=1}^\infty$ is connected,  
then there are $k\in \mathbb{N}$,  an subsequence $\{G_m\}_{m=1}^\infty$ of $\{G_n\}_{n=1}^\infty$ and partitions $\{ H_m^i\}_{i=1}^k$ of $G_m$ for all $m$ such that $\{H^i_m\}_{m=1}^\infty$ are sequences of expanders for all $i=1,2,\dots ,k$.  
\end{corollary}

\begin{proof}
For a sequence of multi-way expanders $\{G_n=(V_n,E_n)\}_{n=1}^\infty$, there exists $k\in \mathbb{N}$ such that $\inf_{n\in \mathbb{N}} h_{k+1}(G_n)>0$ and $\inf_{n\in \mathbb{N}} h_k(G_n)=0$.  
Take a subsequence $\{G_m\}_{m=1}^\infty$ such that $h _k(G_m)\to 0$ as $m\to \infty$. 
Theorem \ref{thm:l-sub} implies that for each $m\in \mathbb{N}$ 
there exists a partition $\{H^i_m=(V^i_m,E^i_m)\}_{i=1}^k$ of $G_m$ such that 
$$\inf_{m\ge m_0} \min_{i=1,2,\dots ,k} h (H^i_m)\ge \inf_{m\ge m_0} h_{k+1}(G_m)/3^{k+1} >0$$ for large $m_0$ enough.
Since at least a sequence $\{H_m^i\}_{m=1}^\infty$ satisfies $|V^i_m| \to \infty \text{ as } m\to \infty$, $\{H_m^i\}_{m=m_0}^\infty$ is a sequence of expanders.
Hence let $H_m^i$ be a connected component in $G_m$ for each $1\le m<m_0$, 
then $\{H_m^i\}_{m=1}^\infty$ is a sequence of expanders.

When each graph in $\{G_m\}_{m=1}^\infty$ is connected, from Theorem \ref{thm:l-sub} we have  
$$\min_{i=1,2,\dots ,k} |V^i_m|
\ge \min_{i=1,2,\dots ,k} \frac{|V^i_m|}{|\partial V^i_m|}
\ge \frac{1}{3^kh_k(G)}\to \infty \text{ as } m\to \infty. $$ 
Hence $\{H^i_m\}_{m=1}^\infty $ are sequences of expanders for all $i=1,2,\dots ,k$.
\end{proof}

\section{Remark for second smallest eigenvalues of the Laplacians of induced subgraphs in a partition}

Using the result by Lee-Gharan-Trevisan, we have the following theorem, which is a similar result of Theorem \ref{thm:l-sub}. 

\begin{corollary}\label{cor:l-sub}
If $\lambda_{k+1}(G)\ge Ck^23^{k+2}\deg(G)^2\sqrt{\lambda_k(G)} $ for some $k\in \mathbb{N}$, where $C$ is a constant in the inequality (\ref{eq:hk=lambdak}) in the result by Lee-Gharan-Trevisan,
then there exists a $k$-partition $\{G^i=(V^i,E^i)\}_{i=1}^k$ of $G$ satisfying
\begin{eqnarray}
\max_{i=1,2,\dots, k} \frac{|\partial V^i|}{|V^i|} \le Ck^2\deg(G) \sqrt{\lambda_k(G)},  \ \ \ 
\lambda_{k+1}(G)\le  3^{k+2} (\deg(G))^{\frac{3}{2}}\sqrt{\min_{i=1,2,\dots ,k}\lambda_2(G^i)}. 
\label{ineq:cor-2}
\end{eqnarray}
\end{corollary}

On the other hand, we have a similar result of Lemma \ref{lem:1}. 
Give inner products on $\mathbb{R}^V$ by $\langle f,g\rangle := \sum_{x\in V}f(x)g(x)$ for $f,g\in \mathbb{R}^V$, 
and on $\mathbb{R}^E$ by $\langle a,b\rangle := \sum_{e\in V}a(e)b(e)$ for $a,b\in \mathbb{R}^E$. 
Let $\|f\|:= \sqrt{\langle f,f\rangle}$ for $f\in \mathbb{R}^V$ and $\|a \|:= \sqrt{\langle a,a \rangle}$ for  $a\in \mathbb{R}^E$. 
Define $d: \mathbb{R}^V\to \mathbb{R}^E$ by $df(xy):= |f(x)-f(y)|$ for $f\in \mathbb{R}^V$and $xy\in E$.  
Then $\langle \Delta_Gf,f\rangle = \|df\|^2$ for $f\in \mathbb{R}^V$, and consequently 
\begin{eqnarray}
\lambda_{k+1}(G) 
&=& 
\sup _{L_k} \inf_{f\in \mathbb{R}^V}\left\{  \frac{\|df\|^2}{\|f\|^2}:f\perp L_k , f\not\equiv 0 \right\}
\label{eq:lk+1}
\\ 
\lambda_k(G) 
&=& 
\inf_{L_k} \sup_{f\in \mathbb{R}^V} \left\{ \frac{\|df\|^2}{\|f\|^2} :f\in  L_k, f\not\equiv 0\right\} 
\label{eq:lk}
\end{eqnarray}
where $L_{k}$ is a $k$-dimensional subspace of $\mathbb{R}^V$.

\begin{lemma}\label{lem:2}
For any partition $\{G^i=(V^i,E^i)\}_{i=1}^k$ of $G$
\begin{equation}
\lambda_{k+1}(G) \ge \min_{i=1,2,\dots ,k}\lambda_2 (G^i) .\label{eq:ll}
\end{equation}
\end{lemma}

\begin{proof}
Let $\psi_0:\equiv 1$ be a constant function on $V$, which is an eigenfunction of $\lambda_1(G)$. 
Let $\psi_i$ be a function on $V$ defined by 
\begin{eqnarray*}
\psi_i (v) = 
\begin{cases}
|V^{i+1}| & \text{if } v \in \cup_{j=1}^iV^j\\
-\sum_{j=1}^i|V^j| & \text{if } v \in V^{i+1}\\
0 & \text{otherwise}
\end{cases}
\end{eqnarray*}
for $i=1,2,\dots ,k-1$. 
Since the functions $\psi_i $ are perpendicular to each other, 
the subspace $P:=\langle \psi_0, \psi_1, \psi_2, \dots, \psi_{k-1}\rangle$ of $\mathbb{R}^V$ is $k$-dimensional. 
Hence we have 
\begin{eqnarray*}
\lambda_{k+1}(G) 
\ge \inf_{f\in \mathbb{R}^V}\left\{  \frac{\|df\|^2}{\|f\|^2}:f\perp P , f\not\equiv 0 \right\} . 
\end{eqnarray*}
A function $f\in \mathbb{R}^V$ with $f\perp P $ satisfies $\sum_{v\in V}f(v)=0$ and 
\begin{eqnarray*}
\sum _{j=1}^{i}\sum_{v\in V^j}f(v)|V^{i+1}| - \sum_{v\in V^{i+1}}f(v)\sum_{j=1}^i|V^j|=0
\end{eqnarray*}
for $i=1,2,\dots , k-1$. 
In particular, when $i=k-1$, 
since $\sum _{j=1}^{k-1}\sum_{v\in V^j}f(v)=-\sum_{v\in V^k}f(v)$, 
we have
\begin{eqnarray*}
0 
&=& 
\sum _{j=1}^{k-1}\sum_{v\in V^j}f(v)|V^k| - \sum_{v\in V^k}f(v)\sum_{j=1}^{k-1}|V^j| 
\\ &=& 
- \sum_{v\in V^k}f(v)|V^k| -  \sum_{v\in V^k}f(v)\sum_{j=1}^{k-1}|V^j| 
\\ &=& 
-\sum_{v\in V^k}f(v)|V|. 
\end{eqnarray*} 
Hence $\sum _{j=1}^{k-1}\sum_{v\in V^j}f(v)=-\sum_{v\in V^k}f(v)=0$.  
Inductively we get  $\sum_{v\in V^i} f(v) =0$ for $i=1,2,\dots ,k$. 
This means $f|_{V^i}\perp \psi_0|_{V^i}$ for every $i=1,2,\dots ,k$. 
Since $\psi_0|_{V^i}$ is an eigenfunction of $\lambda_1(G^i)$, using (\ref{eq:lk}) for each $\lambda_2(G^i)$ we obtain
\begin{eqnarray*}
\| df \|^2 
&=& \sum _{e\in E^1} |df(e)|^2 + \dots + \sum _{e\in E^k} |df(e)|^2 + \sum _{e\in E-\cup_{i=1}^kE^i} |df(e)|^2 \\ 
&\ge & \lambda _2(G^1) \sum _{v\in V^1}|f(v)|^2 + \dots + \lambda _2(G^k) \sum _{v\in V^k}|f(v)|^2 \\
&\ge & \min_{i=1,2,\dots ,k}\lambda_2 (G^i)  \|f\|^2
\end{eqnarray*}
for any $f\in \mathbb{R}^V$ with $f\perp P $.
This implies $\lambda _{k+1}(G)\ge \min_{i=1,2,\dots ,k}\lambda_2 (G^i) $.  
\end{proof}

\section{Coarse non-embeddability of sequences of higher order expanders}\label{sec:CE}

\begin{definition}[\cite{MR1253544}]
A sequence of metric spaces $\{X_n\}_{n=1}^\infty $ is said to be {\it coarsely embeddable} into a sequence of metric space $\{Y_n\}_{n=1}^\infty $ if there exist two non-decreasing functions $\rho_1$ and $\rho_2$ on $[0,+\infty)$ and maps $\{f_n:X_n\to Y_n\}_{n=1}^\infty $ such that 
\begin{enumerate}
\item $\rho_1 (d_{X_n}(x,y))\le d_{Y_n}(f_n(x), f_n(y) ) \le \rho _2(d_{X_n}(x, y)) $ for all $x,y\in X_n$ and $n$; 
\item $\lim_{r\to \infty }\rho_1(r) = + \infty$. 
\end{enumerate}
In particular, for a metric space $Y$, if $Y_n=Y$ for all $n$ and $\{ X_n \}_{n=1}^\infty $ is coarsely embeddable into $\{ Y_n \}_{n=1}^\infty $, then $\{ X_n \}_{n=1}^\infty $ is said to be coarsely embeddable into $Y$. 
Furthermore, if for a metric space $X$, if $X_n=X$ for all $n$ and $\{ X_n \}_{n=1}^\infty $ is coarsely embeddable into $Y$, then $X$ is said to be coarsely embeddable into $Y$. 
\end{definition}

We can endow a graph $G$ with the path metric $d_G(x,y)$ between vertices $x$ and $y$ which is the minimum number of edges in any path connecting $x$ and $y$. 
If a sequence of expanders is coarsely embeddable into a metric space, then the metric space is not coarsely embeddable into any Hilbert space. 
But a sequence of graphs into which a sequence of expanders is embedded as subgraphs may be coarsely embeddable into a Hilbert space. 
For example, for a sequence of expanders $\{H_n\}_{n=1}^\infty$, let $G_n=(V_n,E_n)$ be $V_n = V_{H_n}\cup \{ x \}$ and $E_n = E_{H_n}\cup \{xy:y\in V_{H_n}\}$, then  $\{G_n\}_{n=1}^\infty $ is coarsely embeddable into a Hilbert space.
However we obtain 

\begin{proposition}\label{Prop:1}
Let $\{G_n=(V_n,E_n)\}_{n=1}^\infty$ be a sequence of graphs with $\sup_{n\in \mathbb{N}}\deg(G_n)<\infty $. 
If $\{G_n\}_{n=1}^\infty$  has a sequence of induced subgraphs $H_n$ of $G_n$ such that 
$\{H_n\}_{n=1}^\infty $ is a sequence of expanders,  
then $\{G_n\}_{n=1}^\infty$ is not coarsely embeddable into any Hilbert space. 
\end{proposition}

The proof is about the same of the proof of that a sequence of expanders is not coarsely embeddable into any Hilbert space (\cite{MR1978492}).  

\begin{proof}
Assume $\{G_n\}_{n=1}^\infty$ is coarsely embeddable into a Hilbert space $\mathcal{H}$. 
Then there exist two non-decreasing functions $\rho_1$ and $\rho_2$ on $[0,+\infty)$ and maps $f_n:G_n\to \mathcal{H}$ such that 
\begin{enumerate}
\item $\rho_1 (d_{G_n}(x,y))\le \|f_n(x)-f_n(y)\| \le \rho _2(d_{G_n}(x, y)) $ for all $x,y\in G_n$ and $n$; 
\item $\lim_{r\to \infty }\rho_1(r) = + \infty$, 
\end{enumerate}
where $d_{G_n}$ is the path metric on $G_n$.
From the assumption in the theorem, there exists a constant $C>0$ such that 
\begin{eqnarray*}
C\le \lambda_2(H_n) 
\le  \inf_{\varphi :V_{H_n}\to \mathbb{R}}\left\{ \frac{\| d\varphi \|^2}{\|\varphi \|^2}:\varphi  \text{ is perpendicular to constant functions} \right\} .
\end{eqnarray*}
Hence the subgraphs $H_n$ satisfy  
$$ C \sum_{x\in V_{H_n}} |\varphi (x)|^2 \le \sum_{e\in E_{H_n}} |d\varphi (e)|^2 = \sum_{xy\in E_{H_n}} |\varphi (x)-\varphi (y)|^2 $$
for any functions $\varphi :V_{H_n}\to \mathbb{R}$ which are perpendicular to constant functions. 
On the other hand, 
\begin{eqnarray*}
\sum_{x,y\in V_{H_n}}|\varphi (x)-\varphi (y)|^2 
\le \sum_{x,y\in V_{H_n}}(|\varphi (x)|+|\varphi (y)|)^2 
\le 4|V_{H_n}| \sum_{x\in V_{H_n}}|\varphi (x)|^2 . 
\end{eqnarray*}
Hence 
$$
\frac{C}{4|V_{H_n}|^2} \sum_{x,y\in V_{H_n}}|\varphi (x)-\varphi (y)|^2 
\le \frac{1}{|V_{H_n}|} \sum_{xy\in E_{H_n}} |\varphi (x)-\varphi (y)|^2 
\le \frac{\deg(H_n)}{|E_{H_n}|} \sum_{xy\in E_{H_n}} |\varphi (x)-\varphi (y)|^2. 
$$
The Hilbert space $\mathcal{H}$ can be represented as $L^2(\Omega ,\mu)$ by some measure space $(\Omega ,\mu)$.  
Translating the functions $f_n(\cdot )(z)$ parallel  we may assume $\sum_{x\in V} f_n(x)(z)=0$ for each $z\in \Omega$, which means $f_n(\cdot )(z)$ is perpendicular to constant functions on $V$. 
Thus we have 
$$\frac{C}{4|V_{H_n}|^2} \sum_{x,y\in V_{H_n}}\|f_n(x)-f_n(y)\|^2 
\le \frac{\deg(H_n)}{|E_{H_n}|} \sum_{xy\in E_{H_n}} \|f_n(x)-f_n(y)\|^2 $$
where $\|\ \ \|$ is the norm on $\mathcal{H}$.
This implies 
$$\frac{1}{|V_{H_n}|^2} \sum_{x,y\in V_{H_n}}\rho_1 (d_{G_n}(x,y))^2 
\le \frac{4\deg(H_n)}{C|E_{H_n}|} \sum_{xy\in E_{H_n}} \rho _2(d_{G_n}(x, y))^2 
= \frac{4\deg(H_n)\rho _2(1)^2}{C}<\infty. $$
We will show the left hand side diverges to $\infty$ as $n\to \infty$.
Let $B_{x}(r):= \{ y\in V_n : d_{G_n}(y,x)\le r \}$ 
for $x\in V_n$ and $r\ge 0$.  
For each $x\in V_{H_n}$,  
$$ 
|\{ y\in V_{H_n}: d_{G_n}(y,x) > r\}| 
= \left| V_{H_n} - B_{x}(r) \right| 
\ge |V_{H_n}| - \sum_{i=0}^r \deg(H_n)^i  
= |V_{H_n}| - \frac{\deg(H_n)^{r+1}-1}{\deg(H_n)-1}. 
$$ 
Let $r_n:= (\log_{\deg(H_n)}|V_{H_n}|)/2-1$, 
then $r_n\to \infty $ as $n\to \infty$. 
We get 
\begin{eqnarray*}
|\{ (x,y) \in V_{H_n} ^2: d_{G_n}(x,y) > r_n\}| 
\ge |V_{H_n}|\left( |V_{H_n}| - \frac{|V_{H_n}|^{1/2}-1}{\deg(H_n)-1} \right) 
\ge |V_{H_n}|^2/2   
\end{eqnarray*}
for large $n$. 
Therefore 
\begin{eqnarray*}
\frac{1}{|V_{H_n}|^2} \sum_{x,y\in V_{H_n}}\rho_1 (d_{G_n}(x,y))^2 
\ge \frac{\rho_1 (r_n)^2}{2} 
\to \infty  
\end{eqnarray*}
as $n \to \infty$. 
This is a contradiction. 
\end{proof}

Using this proposition and Corollary \ref{cor:1}, we obtain  

\begin{corollary}\label{cor:CE}
A sequence of multi-way expanders is not coarsely embeddable into any Hilbert space. 
\end{corollary}

\begin{remark}
Whatever a sequence of connected finite graphs is not coarsely embeddable into any Hilbert space, it may not be a sequence of multi-way expanders.
For example, let $\{H_n\}_{n=1}^\infty $ be a sequence of expanders. 
Pick up points $x_n$ and $y_n$ in each $H_n$ arbitrarily. 
Let $\{G_n=(V_n,E_n)\}_{n=1}^\infty$ be a sequence of graphs such that 
$V_n:= \cup_{m=1}^nV_{H_m}$ and $E_n:= \cup_{m=1}^nE_{H_m}\cup \{x_my_{m-1}:m=2,3,\dots n \}$.	
The sequence $\{G_n\}_{n=1}^\infty$ satisfies the assumption of Proposition \ref{Prop:1}, hence it  
is not coarsely embeddable into any Hilbert space. 
On the other hand, for even $n\in \mathbb{N}$ and for $m=1,2,\dots n$, define $f_m:V_{2n}\to \mathbb{R}$ by $f_m(x)=1/\sqrt{|V_{H_{2m}}|}$ if $x\in H_{2m}$ and $f_m(x)=0$ otherwise. 
Then $f_m$ are perpendicular to each other, $\|f_m \|=1$, $\|df_m\|^2 \le 2/|V_{H_{2m}}| $, and $\{f_m\}_{m=1}^{n/2}$ spans $n/2$-dimensional linear space. 
Since 
\begin{eqnarray*}
\lambda_{n/2}(G_{2n}) 
&=& \inf_{L_{n/2}} \sup_f \left\{ \frac{\|df\|^2}{\|f\|^2} :f\in  L_{n/2}, f\not\equiv 0\right\} 
\\ &\le & \sup \left\{ \| df_m\|^2 : m=1,2, \dots ,n/2 \right\} 
\\ &\le & \sup \left\{ \frac{2}{|V_{H_{2m}}|} : m=1,2, \dots ,n/2  \right\} 
\\ &\to & 0 \ \ {\text as }\ \ n\to \infty.
\end{eqnarray*}
Hence $\lambda_k(G_{2n})$ converges to 0 as $n\to \infty$ for any $k\in \mathbb{N}$. 
\end{remark}

Beside Proposition \ref{Prop:1}, we have 

\begin{proposition}\label{Prop:2}
Let $\{G_n\}_{n=1}^\infty$ be a sequence of graphs which has a sequence of induced subgraphs $H_n$ of $G_n$ such that $\{H_n\}_{n=1}^\infty $ is a sequence of expanders. 
If $|\partial V_{H_n}|/|V_{H_n}|\to 0$ as $n\to \infty$, 
then $\{G_n\}_{n=1}^\infty$ is not coarsely embeddable into any Hilbert space. 
\end{proposition}

\begin{proof}
As the proof of Proposition \ref{Prop:1}, it is only necessary to show   
\begin{eqnarray*}
|\{ (x,y) \in V_{H_n} ^2: d_{G_n}(x,y) > r_n\}| 
\ge |V_{H_n}|^2/2   
\end{eqnarray*}
for large $n$, for some $\{r_n\}_{n=1}^\infty$ with $r_n\to \infty $ as $n\to \infty$. 

Let $F_n:=\{ x\in V_{H_n} : xy\in \partial V_{H_n} $ for some $y\}$ and 
$ W_n(r):=\cap_{y\in F_n}\{ x\in V_{H_n}: d_{G_n}(y,x)> r\}$. 
Then 
$$W_n(r)= \cap_{y\in F_n}\{ x\in V_{H_n}: d_{H_n}(y,x)> r\} .$$
Let  
$ W_n(r,z):=\{ w\in W_n(r): d_{G_n}(z,w)> r\} $ for $z \in W_n(r)$.  
Then 
$$W_n(r,z) = \{ w\in W_n(r): d_{H_n}(z,w)> r\},$$
because $d_{H_n}(z,w) \ge d_{G_n}(z,w)$ and if $d_{H_n}(z,w) > d_{G_n}(z,w)$ then the path attaining the distance between $z$ and $w$ with respect to $d_{G_n}$ should be through a point in $F_n$.  
We have $|\{x\in V_{H_n}: d_{H_n}(x,y)\le r\}|\le \deg(H_n)^r$ for $y\in V_{H_n}$.
Since $|F_n| \le |\partial V_{H_n}| $, we have 
\begin{eqnarray*}
|W_n(r)| &\ge & |W_n(r,z)| 
\\ &=& |V_{H_n} - \cup_{y\in F_n}\{ x\in V_{H_n}: d_{H_n}(x,y)\le r\} - \{ x\in V_{H_n}:  d_{H_n}(z,x)\le r\} |
\\ &\ge& |V_{H_n}| - (|F_n|+1)\deg(H_n)^r 
\\ &\ge& |V_{H_n}|\left(1 - \frac{|\partial V_{H_n}|+1}{|V_{H_n}|}\deg(H_n)^r\right) . 
\end{eqnarray*}
We may assume $|\partial V_{H_n}|>0$ for all $n$. 
Since $ (|\partial V_{H_n}|+1)/|V_{H_n}| \to 0$ as $n \to \infty$, 
the sequence $$r_n:= -\frac{1}{2}\log_{\deg(H_n)} \frac{|\partial V_{H_n}|+1}{|V_{H_n}|}$$ diverges to $\infty$ as $n \to \infty$. 
Hence 
\begin{eqnarray*}
|\{ (x,y) \in V_{H_n} ^2: d_{G_n}(x,y) > r_n\}| 
&\ge&|\{(z,w)\in V_{H_n}^2: z\in W_n(r_n), w\in W_n(r_n,z)\}|
\\ &\ge & |V_{H_n}|^2\left(1 - \frac{|\partial V_{H_n}|+1}{|V_{H_n}|}\deg(H_n)^{r_n}\right) ^2
\\ &\ge & \frac{|V_{H_n}|^2}{2} 
\end{eqnarray*}
for large $n$. 
\end{proof}

\begin{acknowledgements}
The author thanks K. Funano for telling me his new result in \cite{MR3061776}, which help to think of results in this paper.
The author was supported by the Grant-in-Aid for the GCOE Program 'Weaving Science Web beyond Particle-Matter Hierarchy', Tohoku University, and GCOE 'Fostering top leaders in mathematics', Kyoto University.
\end{acknowledgements}

\vspace{5mm}
\noindent 
M. Tanaka,\\ 
Advanced Institute for Materials Research, Tohoku University, Sendai, 980-8577 Japan\\ 
E-mail: mamoru.tanaka@wpi-aimr.tohoku.ac.jp

\begin{bibdiv}
\begin{biblist}
\bib{MR782626}{article}{
   author={Alon, N.},
   author={Milman, V. D.},
   title={$\lambda_1,$ isoperimetric inequalities for graphs, and
   superconcentrators},
   journal={J. Combin. Theory Ser. B},
   volume={38},
   date={1985},
   number={1},
   pages={73--88},
   issn={0095-8956},
   review={\MR{782626 (87b:05092)}},
   doi={10.1016/0095-8956(85)90092-9},
}
\bib{MR1989434}{book}{
   author={Davidoff, Giuliana},
   author={Sarnak, Peter},
   author={Valette, Alain},
   title={Elementary number theory, group theory, and Ramanujan graphs},
   series={London Mathematical Society Student Texts},
   volume={55},
   publisher={Cambridge University Press},
   place={Cambridge},
   date={2003},
   pages={x+144},
   isbn={0-521-82426-5},
   isbn={0-521-53143-8},
   review={\MR{1989434 (2004f:11001)}},
   doi={10.1017/CBO9780511615825},
}
\bib{MR743744}{article}{
   author={Dodziuk, Jozef},
   title={Difference equations, isoperimetric inequality and transience of
   certain random walks},
   journal={Trans. Amer. Math. Soc.},
   volume={284},
   date={1984},
   number={2},
   pages={787--794},
   issn={0002-9947},
   review={\MR{743744 (85m:58185)}},
   doi={10.2307/1999107},
}
\bib{MR3061776}{article}{
    AUTHOR = {Funano, Kei and Shioya, Takashi},
     TITLE = {Concentration, {R}icci {C}urvature, and {E}igenvalues of
              {L}aplacian},
   JOURNAL = {Geom. Funct. Anal.},
  FJOURNAL = {Geometric and Functional Analysis},
    VOLUME = {23},
      YEAR = {2013},
    NUMBER = {3},
     PAGES = {888--936},
      ISSN = {1016-443X},
     CODEN = {GFANFB},
   MRCLASS = {Preliminary Data},
  MRNUMBER = {3061776},
       DOI = {10.1007/s00039-013-0215-x},
       URL = {http://dx.doi.org/10.1007/s00039-013-0215-x},
}
\bib{MR1253544}{article}{
   author={Gromov, M.},
   title={Asymptotic invariants of infinite groups},
   conference={
      title={Geometric group theory, Vol.\ 2},
      address={Sussex},
      date={1991},
   },
   book={
      series={London Math. Soc. Lecture Note Ser.},
      volume={182},
      publisher={Cambridge Univ. Press},
      place={Cambridge},
   },
   date={1993},
   pages={1--295},
   review={\MR{1253544 (95m:20041)}},
}

\bib{MR1978492}{article}{
    AUTHOR = {Gromov, M.},
     TITLE = {Random walk in random groups},
   JOURNAL = {Geom. Funct. Anal.},
  FJOURNAL = {Geometric and Functional Analysis},
    VOLUME = {13},
      YEAR = {2003},
    NUMBER = {1},
     PAGES = {73--146},
      ISSN = {1016-443X},
     CODEN = {GFANFB},
   MRCLASS = {20F65 (20F67 20P05 60G50)},
  MRNUMBER = {1978492 (2004j:20088a)},
MRREVIEWER = {Thomas Delzant},
       DOI = {10.1007/s000390300002},
       URL = {http://dx.doi.org/10.1007/s000390300002},
}
		
\bib{MR2247919}{article}{
   author={Hoory, Shlomo},
   author={Linial, Nathan},
   author={Wigderson, Avi},
   title={Expander graphs and their applications},
   journal={Bull. Amer. Math. Soc. (N.S.)},
   volume={43},
   date={2006},
   number={4},
   pages={439--561 (electronic)},
   issn={0273-0979},
   review={\MR{2247919 (2007h:68055)}},
   doi={10.1090/S0273-0979-06-01126-8},
}

\bib{MR2961569}{collection}{
    AUTHOR = {Lee, James R. and Oveis Gharan, Shayan and Trevisan, Luca},
     TITLE = {Multi-way spectral partitioning and higher-order {C}heeger
              inequalities},
 BOOKTITLE = {S{TOC}'12---{P}roceedings of the 2012 {ACM} {S}ymposium on
              {T}heory of {C}omputing},
     PAGES = {1117--1130},
 PUBLISHER = {ACM},
   ADDRESS = {New York},
      YEAR = {2012},
   MRCLASS = {05C40 (05C50 05C85)},
  MRNUMBER = {2961569},
       DOI = {10.1145/2213977.2214078},
       URL = {http://dx.doi.org/10.1145/2213977.2214078},
}


\bib{MR1728880}{article}{
   author={Yu, Guoliang},
   title={The coarse Baum-Connes conjecture for spaces which admit a uniform
   embedding into Hilbert space},
   journal={Invent. Math.},
   volume={139},
   date={2000},
   number={1},
   pages={201--240},
   issn={0020-9910},
   review={\MR{1728880 (2000j:19005)}},
   doi={10.1007/s002229900032},
}
\end{biblist}
\end{bibdiv}
\end{document}